\documentclass[11pt]{amsart}  
\usepackage{amsthm, amsmath, amscd, amssymb, latexsym, stmaryrd, color}
\usepackage[top=3.0cm, bottom=3.0cm, left=2.8cm, right=2.8cm]{geometry}
\theoremstyle{plain}
\newtheorem*{thm}{Main Theorem}
\newtheorem{theorem}{Theorem}[section]
\newtheorem{lemma}[theorem]{Lemma}
\newtheorem{proposition}[theorem]{Proposition}

\newtheorem{corollary}[theorem]{Corollary}

\theoremstyle{definition}

\newtheorem{definition}[theorem]{Definition}

\newtheorem{notation}[theorem]{Notation}

\theoremstyle{remark}
\newtheorem*{ack}{Acknowledgement}
\usepackage[mathscr]{eucal}
\usepackage{graphics, graphpap}
\usepackage{array, tabularx, longtable}
\usepackage{color}
\usepackage{verbatim}
\usepackage[all,cmtip]{xy}
\numberwithin{equation}{section}

\def\gcd{\mathrm{gcd}}

\def\ord{\mathrm{ord}}

\title[Igusa's local zeta function of Thom-Sebastiani type functions]{The Igusa local zeta function\\ of certain Thom-Sebastiani type functions} 

\author[Q.T. L\^e]{Quy Thuong L\^e}
\address{VNU University of Science, Vietnam National University, Hanoi \newline \indent 
	334 Nguyen Trai Street, Thanh Xuan District, Hanoi, Vietnam}
\email{leqthuong@gmail.com}
\thanks{}

\author[H.L. Nguyen]{Hoang Long Nguyen}
\address{VNU University of Science, Vietnam National University, Hanoi \newline \indent 
	334 Nguyen Trai Street, Thanh Xuan District, Hanoi, Vietnam}
\email{nguyenhoanglong\_t66@hus.edu.vn}
\thanks{}


\begin{document}           
	\begin{abstract}
	In this paper, we give an explicit formula of the Igusa local zeta function of a Thom-Sebastiani type sum of two separated-variable Newton non-critical polynomials. Data for the description are available on their Newton polyhedra.

	\end{abstract}
	\maketitle                 

	
\section{Introduction}
	
Let $p$ be a prime, and let $K$ be a $p$-adic field whose ring of integral numbers is $\mathcal O=\mathcal O_K$. Fix a uniformizing parameter $\pi$ of $\mathcal O$. Assume $\mathbb F_q=\mathcal O/(\pi)$. Let $v=\ord_{\pi}$ be the valuation map with $v(0)=\infty$, and define $|\xi|=q^{-v(\xi)}$ for any $\xi\in K^*$. Since $K$ is a locally compact totally disconnected group, there is a unique canonical Haar measure $|dx|$ of $(K^n,+)$ normalized by $|dx|(\mathcal O^n)=1$.

Let $f(x)$ be a polynomial over $\mathcal O$ in $n$ variables $x=(x_1,\dots,x_n)$. Consider the Igusa local zeta function of $f$ defined as follows
$$Z(f;s):=\int_{\mathcal O^n}|f(x)|^s|dx|.$$
It is a fact that $Z(f;s)$ is a rational function in $q^{-s}$. The Igusa local zeta function is an important object in several branch of mathematics such as number theory and singularity theory. There are conjectures such as the the holomorphy conjecture and the monodromy conjecture relate the poles of $Z(f;s)$ to the monodromy of singular point of the hypersurface defined by $f$.  

Our motivation in this topic is to prove that if the $p$-adic monodromy conjecture is true for $f(x)$ and $g(y)$ in separated variables $x$ and $y$, then the conjecture is also true for the function 
$$(f\oplus g)(x,y)=f(x)+g(y).$$ 
In \cite[Section 5.1]{Denef91}, this problem is mentioned and suggested to solve using resolution of singularity and exponential sums. However, there is the lack of an explicit description of $Z(f\oplus g;s)$ in terms of $Z(f;s)$, $Z(g;s)$ and data in the common vanishing of $f$ and $g$. This work not only gives a geometric solution to this problem without mentioning the exponential sum but also helps to answer the similar question on the holomorphy conjecture.  

In this paper, under the Newton non-criticality condition for polynomials $f$ and $g$, we can describe $Z(f \oplus g ;s)$ completely explicitly in terms of $1$-degree $1$-variable polynomials $c_{a,b}(s)$ constructed from the Newton polyhedra of $f$ and $g$. Namely, we prove the following theorem.
	
\begin{thm}[Theorem \ref{maintheorem}]
	Let $f(x) \in \mathcal O[x], g(y) \in \mathcal O [y]$ be two Newton non-critical polynomials. Then the Igusa local zeta function of $f \oplus g$ is described as follows
	\begin{equation*}
	Z(f \oplus g ;s) = \dfrac{P(q^{-s})}{(1 - q^{-1-s})\prod_{a,b}(1 - q^{c_{a,b}(s)})},
	\end{equation*}
	where the product runs over all $(a,b)$ with $a$ (resp. $b$) primitive normal vectors of the facets of the Newton polyhedron of $f$ (resp. $g$), and $P$ is a $1$-variable rational polynomial.
\end{thm}

The notation and techniques we use in this paper partly come from Zuniga-Galindo's works on the Igusa stationary phase formula (see \cite{Zuniga-Galindo2}).


	




	
	

\section{Igusa's stationary phase formula}
\subsection{Zuniga-Galindo's setting}
Let us recall the notation and results about Igusa's stationary phase formula mentioned in \cite{Zuniga-Galindo2}.
Let $f(x)$ be a polynomial in $\mathcal{O}[x]$, and let $(P_k)_{k\geq 1}$ be a sequence of points in $\mathcal{O}^n$. Define a sequence $(f_{P_1,\dots,P_k})_{k\ge1}$ of polynomials in $\mathcal O[x]$ as follows 
\begin{align*}
f_{P_1}(x)&= \pi^{-e_{P_1}(f)}f(P_1 + \pi x),\\
f_{P_1,\dots,P_{k+1}}(x)&= \pi^{-e_{P_1,\dots,P_{k+1}}(f)}f_{P_1,\dots,P_k}(P_{k+1} + \pi x),\ k\ge 1,
\end{align*}
where $e_{P_1}(f)$ (resp. $e_{P_1,\dots,P_{k+1}}(f)$) denotes the smallest valuation of the coefficients in the $x$-polynomial $f(P_1 + \pi x)$ (resp. $f_{P_1,\dots,P_k}(P_{k+1} + \pi x)$). Fix a lifting $R$ of $\mathbb F_q$ in $\mathcal O$, that is, $R$ is mapped bijectively onto $\mathbb F_q$ along the canonical projection. For any subset $\overline{D}$ of $\mathbb F_q^n$, let $D$ denote the preimage of $\overline{D}$ under the canonical projection $\mathcal O^n\to\mathbb F_q^n$. Put
\begin{equation*}
S(f,D):= \{(x_1,\dots,x_n) \in R^n \mid (\overline{x}_1,\dots,\overline{x}_n) \text{ is a singular point of }\overline{f}\text{ in }\overline{D}\},
\end{equation*}
where $\overline{f}(x) \in \mathbb F_q[x]$ is the reduction modulo $\pi$ of $f(x)$. For simplicity, we also define
\begin{align*}
\nu(\bar{f},D)&: = q^{-n}\#\{\overline{P} \in \overline{D}\mid \overline{f}(\overline{P}) \ne 0\},\\
\sigma(\overline{f},D)&: = q^{-n}\#\{\overline{P} \in \overline{D}\mid \overline{f}(\overline{P}) = 0, \overline{P} \text{ is a smooth point of }\overline{f}\}.
\end{align*}
Then, Igusa's Stationary Phase Formula can be written as follows
\begin{align*}
\int_D |f(x)|^s |dx| = \nu(\bar{f},D) + \sigma(\bar{f},D)\dfrac{(1 - q^{-1})q^{-s}}{1 - q^{-1-s}} + \sum_{P \in S(f,D)}q^{-n - e_p(f)s}\int_{\mathcal O^n}|f_P(x)|^s|dx|.
\end{align*}
One also considers the sequence of sets $I_k=I_k(f,D)$, $k\ge1$, defined by
\begin{align*}
I_1&= S(f,D),\\
I_{k + 1}&= \left\{(P_1,\dots,P_{k+1})\mid (P_1,\dots,P_k)\in I_k, P_{k+1} \in S(f_{P_1,\dots,P_k},\mathcal{O}^n)\right\},\ k\ge 1,
\end{align*}
which arises naturally when applying consecutively the stationary phase formula. 
\begin{notation}(Zuniga-Galindo \cite{Zuniga-Galindo1})
Let $f\in \mathcal O[x]$, and let $P = (a_1,\dots,a_n) \in \mathcal O^n$. If $P$ is not a singular point of $f$, put
\begin{equation*}
L(f,P):= \min \left\{v\left(f(P)\right),v\left(\dfrac{\partial f}{\partial x_1}(P)\right),\dots,v\left(\dfrac{\partial f}{\partial x_n}(P)\right)\right\}.
\end{equation*}
If $P$ is not a critical point of $f$, i.e. there exists an $i\in \{1,\dots,n\}$ such that, $\frac{\partial f}{\partial x_i}(P) \not= 0$, then one puts
\begin{equation*}
\ell(f,P):=\min \left\{v\left(\dfrac{\partial f}{\partial x_1}(P)\right),\dots,v\left(\dfrac{\partial f}{\partial x_n}(P)\right)\right\}.
\end{equation*}
\end{notation}

\begin{proposition}[Zuniga-Galindo \cite{Zuniga-Galindo2}]
Let $D \subseteq \mathcal O^n$ be a compact set, and let $f \in \mathcal O[x]$ be a polynomial such that $f$ does not have any singular points in $D$. Then 
\begin{equation*}
C(f,D):= \sup_{P \in D} L(f,P) < + \infty.
\end{equation*}
\end{proposition}
\begin{corollary}
Let $D \subseteq \mathcal O^n$ be a compact set, and $f(x) \in \mathcal O[x]$ be a polynomial such that $f$ does not have any critical points in $D$. Then 
\begin{equation*}
c(f,D):= \sup_{P \in D} \ell(f,P) < + \infty. 
\end{equation*}
\end{corollary}


\begin{lemma}[Zuniga-Galindo, \cite{Zuniga-Galindo2}]\label{ZG2}
Let $\overline{D} \subseteq \mathbb F_q^n$ and $f\in \mathcal{O}[x]$. Assume that $f$ has no singular points in $D$. Then the following hold.
\begin{itemize}
\item[(i)] If $(P_1,\dots,P_{k+1}) \in I_{k+1}(f,D)$, then 
$$L(f_{P_1,\dots,P_k},0) \le L(f,P_1 + \pi P_2 +\dots + \pi^{k-1}P_k) - k.$$

\item[(ii)] If $k > C(f,D)  + 1$, then $I_k = \varnothing$.
\end{itemize}
\end{lemma}

\subsection{A consecutive stationary phase formula}
Let $f$ be in $\mathcal O[x]$ that has no singular points in $D$. It implies from Lemma \ref{ZG2} that 
$$m(f,D):=\max\{k\ge 1\mid I_k \ne \varnothing\}<+\infty.$$
For simplicity, we write $m$ for $m(f,D)$, thus for $(P_1,\dots,P_m) \in I_m$, we have $S(f_{P_1,\dots,P_m},\mathcal{O}^n) = \varnothing$. Applying the stationary phase formula $m$ times, we obtain
\begin{equation}\label{expansion}
	\begin{split}
		\int_D |f(x)|^s|dx| &= \nu(\bar{f},D) + \sigma(\bar{f},D)\dfrac{(1 - q^{-1})q^{-s}}{1 - q^{-1-s}}\\
		&\quad + \sum_{k = 1}^mq^{-kn}\left(\sum_{(P_1,\dots,P_k)\in I_k}\nu(\bar{f}_{P_1,\dots,P_k},\mathcal O^n)q^{-E_{P_1,\dots,P_k}(f)s}\right)\\
		&\quad +\dfrac{(1 - q^{-1})q^{-s}}{1 - q^{-1-s}}\sum_{k = 1}^mq^{-kn}\left(\sum_{(P_1,\dots,P_k)\in I_k}\sigma(\bar{f}_{P_1,\dots,P_k},\mathcal O^n)q^{-E_{P_1,\dots,P_k}(f)s}\right),
	\end{split}
\end{equation}
where 
$$E_{P_1,\dots,P_k}(f) = \sum_{i = 1}^k e_{P_1,\dots,P_i}(f).$$ 
This formula gives us the following lemma.
	
\begin{lemma}[Zuniga-Galindo \cite{Zuniga-Galindo2}] \label{first-important}
Let $\overline{D} \subseteq \mathbb F_q^n$ and $f\in \mathcal{O}[x]$. Assume that $f$ has no singular points in $D$. Then 
\begin{equation*}
\int_D |f(x)|^s |dx| = \dfrac{P(q^{-s})}{1 - q^{-1-s}},
\end{equation*}
where $P$ is a $\mathbb Q$-polynomial in one variable.
\end{lemma}

By some slight modifications, we obtain a similar version of \cite[Lemma 2.5]{Zuniga-Galindo2}.

\begin{proposition}\label{ezlemma}
Let $\overline{D} \subseteq \mathbb F_q^n$ and $f\in \mathcal{O}[x]$. Assume that $f$ has no singular points in $D$. Let $F(x) = f(x) + \pi^\beta g(x)$, where $g \in \mathcal O[x]$, and $\beta > C(f,D) + m(f,D)$. Then
\begin{equation*}
\int_D |f(x)|^s|dx| = \int_D |F(x)|^s |dx|.
\end{equation*}
\end{proposition}

\begin{proof}
It suffices to prove the following statements:
\begin{itemize}
\item[(i)] $I_k(f,D) = I_k(F,D)$;
\item[(ii)] $E_{P_1,\dots,P_k}(f) = E_{P_1,\dots,P_k}(F)$ for all $(P_1,\dots,P_k) \in I_k(f,D) = I_k(F,D)$;
\item[(iii)] $\overline{f}_{P_1,\dots,P_k}(x) = \overline{F}_{P_1,\dots,P_k}(x)$ for all $(P_1,\dots,P_k) \in I_k(f,D) = I_k(F,D)$.
\end{itemize}
The item (i) can be proved easily by induction on $k\geq 1$. Now, taking any point $(P_1,\dots,P_k)$ in $I_{k}(f,D) = I_{k}(F,D)$ we have 
\begin{align*}
	\pi^{E_{P_1,\dots,P_k}(F)}F_{P_1,\dots,P_k}(x) &= F(P_1 + \pi P_2 + \dots + \pi^{k-1}P_k + \pi^kx)\\
	&= f(P_1 + \pi P_2 + \dots + \pi^{k-1}P_k + \pi^kx) + \pi^\beta g_k(x)\\
	&= \pi^{E_{P_1,\dots,P_k}(f)}f_{P_1,\dots,P_k}(x) + \pi^\beta g_k(x),
\end{align*}
where 
$$g_k(x) = g(P_1 + \pi P_2 + \dots + \pi^{k-1}P_k + \pi^kx).$$ 
Since $\pi$ does not divide the least common divisor of all the coefficients of $f_{P_1,\dots,P_k}(x)$, the number $E_{P_1,\dots,P_k}(f)$ is the minimum $\pi$-order of all the coefficients of $ f(P_1 + \pi P_2 + \dots + \pi^{k-1}P_k + \pi^kx)$. Taking the Taylor expansion of the latter at $Q = P_1 + \pi P_2 + \dots + \pi^{k-1}P_k \in D$ gives
\begin{align*}
	f(P_1 + \pi P_2 + \dots + \pi^{k-1}P_k + \pi^kx)= f(Q) + \sum_{i = 1}^n\dfrac{\partial f}{\partial x_i}(Q)\pi^kx_i + (\text{higher terms}).
\end{align*}
It implies that
\begin{align*}
	E_{P_1,\dots,P_k}(f) &\le \min\left\{v\left(f(Q)\right),v\left(\dfrac{\partial f}{\partial x_1}(Q)\pi^k\right),\dots,v\left(\dfrac{\partial f}{\partial x_n}(Q)\pi^k\right)\right\}\\
	&\le L(f,Q) + k \\
	&\le C(f,D) + m(f,D)\\
	& < \beta.
\end{align*}
If  $E_{P_1,\dots,P_k}(F) < E_{P_1,\dots,P_k}(f) < \beta$, then
\begin{equation*}
	F_{P_1,\dots,P_k}(x) = \pi^{E_{P_1,\dots,P_k}(f) - E_{P_1,\dots,P_k}(F)}f_{P_1,\dots,P_k}(x) + \pi^{\beta - E_{P_1,\dots,P_k}(F)}  g_k(x),
\end{equation*}
which is a contradiction as the greatest common divisor of the coefficients of $F_{P_1,\dots,P_k}$ is also not a multiple of $\pi$. Thus $E_{P_1,\dots,P_k}(F) \ge E_{P_1,\dots,P_k}(f)$. The same argument give the reverse inequality, which proves that 
$$E_{P_1,\dots,P_k}(F) = E_{P_1,\dots,P_k}(f).$$ 
It implies that 
\begin{equation*}
	F_{P_1,\dots,P_k}(x) = f_{P_1,\dots,P_k}(x) + \pi^{\beta - E_{P_1,\dots,P_k}(f)}  g_k(x),
\end{equation*}
hence
\begin{equation*}
\overline{f}_{P_1,\dots,P_k}(x) = \overline{F}_{P_1,\dots,P_k}(x).
\end{equation*}
Since the previous equality holds for all $(P_1,\dots,P_m) \in I_m(f,D)$, where $m=m(f,D)$, we get
$$S(\overline{F}_{P_1,\dots,P_{m}}, \mathcal O^n) =S(\overline{f}_{P_1,\dots,P_{m}}, \mathcal O^n) = \varnothing.$$ 
Therefore, $m(F,D) = m(f,D)$, the lemma then follows from the expansion (\ref{expansion}).
\end{proof}

\begin{corollary}\label{second-important}
Let $\overline{D} \subseteq \mathbb F_q^n$ and $f\in \mathcal{O}[x]$. Assume that $f$ has no critical points in $D$. Let $F(x) = f(x) + \pi^\beta g(x)$, where $g(x) \in \mathcal O[x]$ and $\beta > 2c(f,D) + 1$. Then
\begin{equation*}
\int_D |f(x)|^s|dx| = \int_D |F(x)|^s |dx|.
\end{equation*}
\end{corollary}
\begin{proof}
By Lemma \ref{ZG2} we get $m(f,D) \le C(f,D) + 1$. The corollary follows from Lemma \ref{ezlemma} and the fact that $ C(f,D)\le c(f,D)$.
\end{proof}

	
\section{Newton polyhedron}
\subsection{Newton non-critical polynomials}
In this section, we briefly recall some well-known results about the Newton polyhedron of a polynomial (see e.g. \cite{Denef-Hoornaert}). Let $f(x) = \sum_{\omega \in \mathbb N^n}a_\omega x^\omega$ be a non zero polynomial over $\mathcal O$ with $f(0) = 0$, where $x^\omega$ stands for $x_1^{\omega_1}\cdots x_n^{\omega_n}$ with $x=(x_1,\dots,x_n)$ and $\omega=(\omega_1,\dots,\omega_n)$. We denote by $\mathbb R^+$ the set of non-negative real numbers. The {\it Newton polyhedron} $\Gamma_f$ of $f$ is the convex hull in $(\mathbb R^+)^n$ of the set
\begin{equation*}
	\bigcup_{\omega \in \mathrm{supp}(f)} \omega + (\mathbb R^+)^n,
\end{equation*}
where $\mathrm{supp}(f) = \{\omega \in \mathbb N^n: a_\omega \ne 0\}$. Each proper face of $\Gamma_f$ is the intersection of $\Gamma_f$ with a supporting hyperplane, and a face of $\Gamma_f$ is a proper face of $\Gamma_f$ or the Newton polyhedron itself. For $a \in (\mathbb R^+)^n$, we define
\begin{equation*}
	m_f(a) = \inf_{x \in \Gamma_f} \{a \cdot x\},
\end{equation*}
and the first meet locus of $a$ is
\begin{equation*}
	F_f(a) = \{x \in \Gamma_f \mid a \cdot x = m_f(a)\}.
\end{equation*}

\begin{definition}
	For each face $\tau$ of $\Gamma_f$, the cone associated to $\tau$ is
	\begin{equation*}
		\Delta_{f,\tau} = \{ a \in (\mathbb R^+)^n\mid F_f(a) = \tau\}.
	\end{equation*}
\end{definition}
From now on, if there is no possibility of confusion, we drop the subscript $f$ and simply write $m(a),F(a),$ and $\Delta_\tau$.

A vector $a \in \mathbb N^n$ is called {\it primitive} if the greatest common divisor of its components is $1$. A {\it facet} of $\Gamma_f$ is a proper face of codimension one. We can prove that for every facet of $\Gamma$, there exists a unique primitive vector in $\mathbb N^n \setminus \{0\}$ perpendicular to that facet. Every proper face $\tau$ of $\Gamma_f$ is a finite intersection of the facets of $\Gamma_f$ containing $\tau$.

\begin{lemma}[Denef-Hoornaert \cite{Denef-Hoornaert}]
	Let $\tau$ be a proper face of $\Gamma_f$, and let $\gamma_1,\dots,\gamma_e$ be the facets of $\Gamma_f$ which contain $\tau$. Let $a_1,\dots,a_e$ be the unique primitive vectors in $\mathbb N^n \setminus \{0\}$ which are perpendicular to respectively $\gamma_1,\dots,\gamma_e$. The the cone $\Delta_\tau$ associated to $\tau$ is the following convex cone
	\begin{equation*}
		\Delta_\tau = \{\lambda_1 a_1 + \dots + \lambda_e a_e\mid  \lambda_i \in \mathbb R,\lambda_i >0\}
	\end{equation*}
and its dimension is equal to $n - \dim\tau$. Moreover, we also have
\begin{equation*}
	\overline{\Delta}_\tau = \{a \in (\mathbb R^+)^n\mid  \tau \subseteq F(a)\} = \{\lambda_1 a_1 + \dots + \lambda_e a_e\mid  \lambda_i \in \mathbb R,\lambda_i \ge0\}.
\end{equation*}
\end{lemma}

\begin{definition}
	We call $\Delta \subseteq \mathbb R^n$ a {\it strictly positively spanned cone} by vectors $a_1,\dots,a_e$ in $\mathbb R^n \setminus \{0\}$ if
	\begin{equation*}
		\Delta = \{\lambda_1a_1 + \dots + \lambda_ea_e \mid \lambda_i \in \mathbb R, \lambda_i>0\}.
	\end{equation*}
	If $a_1,\dots,a_e$ are linearly independent over $\mathbb R$, $\Delta$ is called a {\it simplicial cone}. If moreover, $a_1,\dots,a_e$ are  in $\mathbb Z^n$, we say $\Delta$ is a {\it rational} simplicial cone.  Furthermore, if $a_1,\dots,a_e$ can be completed to be a basis of $\mathbb Z^n$, we call $\Delta$ a simple cone.
\end{definition}

\begin{lemma}[Denef-Hoornaert \cite{Denef-Hoornaert}]
	Let $\Delta \subseteq \mathbb R^n$ be the strictly positively spanned cone by vectors $a_1,\dots,a_e$ in $\mathbb R^n \setminus \{0\}$. Then there exists a finite partition of $\Delta$ into cones $\Delta_i$, such that each $\Delta_i$ is strictly positively spanned by some vectors from the set $\{a_1,\dots,a_e\}$ which are linearly independent over $\mathbb R$.
\end{lemma}

\begin{definition}
Let $f(x)$ be a polynomial in $\mathcal O[x]$. Then $f$ is called {\it Newton non-critical} if for every face $\tau$ of $\Gamma_f$, the following system
\begin{equation*}
\dfrac{\partial f_\tau}{\partial x_i}(x) = 0, \hspace{10pt}\text{for }1\leq i\leq n,
\end{equation*}
has no solution in $(K^\times)^n$.
\end{definition}


\subsection{A technical lemma}\label{Sec32}
Given $c,d$ be positive integers, we introduce the following map
\begin{equation*}
	\phi_{c,d}: \mathbb Z_{>0}^2 \to \mathbb Z_{>0}^2
\end{equation*}
defined by
\begin{equation*}
		\phi_{c,d}(s,t) = \begin{cases}
		(c,d)&\text{if}\ s = t\\
		(s-t,d)&\text{if}\ s>t\\
		(c, t - s)&\text{if}\ s<t.
\end{cases}
\end{equation*}
We put
\begin{align*}
(c_1,d_1) &= (c,d),\\
(c_{k+1},d_{k+1}) &= \phi_{c,d}(c_k,d_k), \ k \ge 1.
\end{align*}
Hosokawa in \cite{Hiroshi-Hosokawa} proved that the sequences $(c_k,d_k)_{k\ge 1}$ is periodic. In order to keep this paper self-contained, let us sketch his proof as follows. We may assume $c \le d$. Write $\phi$ for $\phi_{c,d}$, and put $c = ke, d = ke'$, where $k = \gcd(c,d)$. Write $e' = qe + r$, where $ 0\le r\le e - 1$. If $r = 0$, then $e = 1$ as $\gcd(e,e') = 1$. In this case, $(c_k,d_k)_{k\geq 1}$ is clearly periodic with period $q$. In the case where $r > 0$, by dividing the closed interval $[0,er]$ into $r$ sub-intervals $I_i := [e(i - 1),ei]$, $1\leq i\leq r$, we put
\begin{equation*}
	\Lambda_0 = \{\lambda\in [0,e-1]\cap\mathbb Z\mid \lambda r \text{ and } (\lambda + 1) r \text{ are not in the same } I_i\}.
\end{equation*}
As $\gcd(e,r) = \gcd(e,e') = 1$, each $\lambda r$ must lies in the interior of some $I_i$, so $|\Lambda_0| = r - 1$. For each integer $n$, we denote by $\overline{n}$ the remainder of $n$ modulo $e$. Note that $(c_1,d_1) = (ke,ke') = (ke, k(qe +r))$, and it takes us $q$ steps to go from $(ke,k(qe+r))$ to $(ke,kr) = (ke, k\bar{r})$ along $\phi$ as follows
\begin{equation*}
	(ke,k(qe + r)) \longmapsto (ke, k(q-1)e + kr)\longmapsto \dots \longmapsto (ke, kr) = (ke, k\overline{r}).
\end{equation*}
For $1\leq \lambda\leq e-2$, we consider the following cases.
\begin{itemize}
	\item If $\lambda \notin \Lambda_0$, then $\overline{(\lambda + 1)r} = r + \overline{\lambda r}$, thus it takes us $(q + 1)$ steps to go from $(ke,k\overline{\lambda r})$ to $(ke, ke\overline{(\lambda + 1)r}$ along $\phi$ as follows
	\begin{align*}
		(ke,k\overline{\lambda r})\longmapsto (k(e - \overline{\lambda r}),ke')\longmapsto & (ke, k(q-1)e + k(r + \overline{\lambda r}))\longmapsto\\
		& \cdots \longmapsto (ke, k(r + \overline{\lambda r})) = (ke,k\overline{(\lambda + 1)r}).
	\end{align*}
	\item If $\lambda \in \Lambda_0$, then $\overline{(\lambda + 1)r} = r + \overline{\lambda r} - e$, thus we need an extra step to go from $(ke,k\overline{\lambda r})$ to $(ke, ke\overline{(\lambda + 1)r}$ along $\phi$ as follows
	\begin{align*}
		(ke,k\overline{\lambda r})\longmapsto &(k(e - \overline{\lambda r}),ke')\longmapsto \cdots\\
		 &\longmapsto (ke, k(r + \overline{\lambda r})) \longmapsto (ke,k(r + \overline{\lambda r} - e )) = (ke,k\overline{(\lambda + 1)r}).
	\end{align*}
\end{itemize}
From $(ke,k\overline{(e-1)r}) = (ke, k(e-r))$ to $(ke, ke')$, it takes us $q + 2$ steps along $\phi$, namely,
\begin{equation*}
	(ke,k(e-r))\longmapsto(kr, ke')\longmapsto(ke,kqe)\longmapsto\cdots \longmapsto (ke,ke)\longmapsto (ke,ke').
\end{equation*}
In summary, going once from $(ke,ke')$ to itself needs the number of steps to be
\begin{equation*}
	q + (q + 1)(e - 2 - |\Lambda_0|) + (q + 2)|\Lambda_0| + q + 2 = e + e' - 1=:p.
\end{equation*}
This number also covers the case $r=0$. So, for any $0\leq r\leq e-1$, we have $(c_1,d_1)=(c_{1 + p},d_{1 + p})$.

Furthermore, the previous arguments have proved the following technical lemma.

\begin{lemma}\label{technicallemma}
	Let $c,d$ be positive integers and put $e = c/\gcd(c,d), e' = d/\gcd(c,d)$. Then $(c_k,d_k)_{k \ge 1}$ is periodic with period $p:= e + e' -1$. Furthermore, if $(\mu_k,\nu_k)_{k\ge 2}$ is defined as
	\begin{equation*}
		(\mu_k,\nu_k) = \begin{cases}
			(\min\{c_k,d_k\}, \tilde{c} + \tilde{d})& \text{if}\ c_k = d_k,\\
			(\min\{c_k,d_k\}, \tilde{d})& \text{if}\ c_k > d_k,\\
			(\min\{c_k,d_k\}, \tilde{c})& \text{if}\ c_k < d_k,
		\end{cases}
	\end{equation*}
where $\tilde{c},\tilde{d}$ be arbitrary integers, then 
\begin{equation*}
	\sum_{k=2}^{p+1}\mu_k = \mathrm{lcm}(c,d) \quad\text{and}\quad \sum_{k=2}^{p+1} \nu_k = e\tilde{d} + e'\tilde{c}.
\end{equation*}
\end{lemma}


\section{The Thom-Sebastiani type sum of non-critical functions}

\subsection{Computations on face functions} For any $f(x_1,\dots,x_n) \in \mathcal O[x_1,\dots,x_n]$ and $A\subseteq \mathbb{R}^n$ we put
\begin{equation*}
E_A := \{(x_1,\dots,x_n) \in \mathcal{O}^n\mid (v(x_1),\dots,v(x_n)) \in A\},
\end{equation*}
and
\begin{equation*}
Z_A(f;s) = \int_{E_A} |f(x)|^s |dx|.
\end{equation*}
When $A=(\mathbb R^+)^n$, $E_A$ is nothing but $\mathcal O^n$. In this case, we write simply $Z(f;s)$ for $Z_{\mathcal O^n}(f;s)$.

Consider polynomials $f(x) \in \mathcal O[x]$ and $g(y) \in \mathcal{O}[y]$, with $x=(x_1,\dots,x_n)$, $y=(y_1,\dots,y_m)$. We also consider the following polynomial
\begin{equation*}
f \oplus g := f(x) + g(y) \in \mathcal O[x,y].
\end{equation*}
For any $a \in \mathbb{N}^n$ and $b \in \mathbb{R}^m$, the setting in the previous section provides us the integers $m_f(a)\geq 0$ and $m_g(b)\geq 0$. 

\begin{definition}
We define $1$-degree polynomials $c_{a,b}(s)$ in $s$ as follows
\begin{itemize}
	\item[(i)] If $m_f(a) > 0$ and $m_g(b)>0$, set 
	$$c_{a,b}(s) := -ls - e|b| - e'|a|,$$ 
	where 
	$$l = \mathrm{lcm}(m_f(a),m_g(b)), \ \ e = \frac{m_f(a)}{\mathrm{gcd}(m_f(a),m_g(b))},\ \ \text{and}\ \ e' = \frac{m_g(b)}{\mathrm{gcd}(m_f(a),m_g(b))}.$$
	\item[(ii)] If $m_f(a) = 0$ or $m_g(b) = 0$, set 
	$$c_{a,b}(s) := -\infty.$$
\end{itemize}
\end{definition}

\begin{proposition}\label{important-computation}
Let $f(x) \in \mathcal O[x]$ and $g(y) \in \mathcal O [y]$ be Newton non-critical polynomials, let $\tau$ and $\gamma$ be any faces of $\Gamma_f$ and $\Gamma_g$, respectively. Suppose that $\{a_1,\dots a_u\}$ and $\{b_1,\dots b_v\}$ are $\mathbb{R}$-linearly independent sets that satisfy the following inclusions 
\begin{align*}
\Delta_1:= \bigoplus_{i = 1}^u \mathbb{N} a_i \subseteq \overline{\Delta}_\tau, \quad\text{and}\quad \Delta_2 := \bigoplus_{j = 1}^v \mathbb{N} b_j \subseteq \overline{\Delta}_\gamma.
\end{align*}
Then 
\begin{equation*}
Z_{\Delta_1 \times \Delta_2}(f_\tau \oplus g_\gamma; s) = \dfrac{P(q^{-s})}{(1 - q^{-1-s})\prod_{i = 1}^u\prod_{j = 1}^v(1 - q^{c_{a_i,b_j}(s)})},
\end{equation*}
where $P$ is a rational polynomial in one variable.
\end{proposition}

\begin{proof}

We prove the proposition by induction on $u+v$.

\textit{The base case} $u = v = 1$. First, considering the case where $c:=m_f(a_1)>0,$ $d:=m_g(b_1) > 0$, we define four sequences $(c_k)_{k \ge 1}, (d_k)_{l \ge 1}, (\mu_k)_{k\ge2}, (\nu_k)_{k\ge2}$ as in Section \ref{Sec32}, with $\tilde{c} = |a_1|$ and $\tilde{d} = |b_1|$. Define 
$$Z_{c,d}(s): = Z_{\Delta_1\times\Delta_2}(\pi^cf_\tau \oplus \pi^dg_\gamma;s).$$

\begin{lemma}\label{ez-base}
Let  $u = v = 1$. For all $k\ge1$,
\begin{equation*}
Z_{c_{k},d_{k}}(s) = \dfrac{Q_k(q^{-s})}{1 - q^{-1-s}} + q^{-\mu_{k+1}s - \nu_{k+1}}Z_{c_{k+1},d_{k+1}}(s),
\end{equation*}
where, for each $k\ge 1$, $Q_k$ is a rational polynomial.
\end{lemma}

\begin{proof}
We verify the lemma for $c_k = d_k$, the rest cases are proved in the same method. We have
\begin{align*}
Z_{c_k,d_k}(s) = &\int_{\{v(x) = 0\}\times\{v(y) = 0\}}|\pi^{c_k}f_\tau(x) +\pi^{d_k}g_\gamma(y)|^s|dxdy|\\
&+ \int_{\{v(x) = 0\}\times\{v(y) = lb_1, l \ge 1\}}|\pi^{c_k}f_\tau(x) +\pi^{d_k}g_\gamma(y)|^s|dxdy|\\
&+ \int_{\{v(x) = ka_1,k\ge1\}\times\{v(y) = 0\}}|\pi^{c_k}f_\tau(x) +\pi^{d_k}g_\gamma(y)|^s|dxdy|\\
&+ \int_{\{v(x) = ka_1,k\ge1\}\times\{v(y) = lb_1,l\ge1\}}|\pi^{c_k}f_\tau(x) +\pi^{d_k}g_\gamma(y)|^s|dxdy|.
\end{align*}
Lemma \ref{first-important} implies that the first integral is of the form $P_1(q^{-s})/(1 - q^{-1-s})$ for some $1$-variable rational polynomial $P_1$. Regarding the second integral, we have
\begin{align*}
\int_{\{v(x) = 0\}\times\{v(y) = lb_1, l \ge 1\}}&|\pi^{c_k}f_\tau(x) +\pi^{d_k}g_\gamma(y)|^s|dxdy|\\
&= \sum_{l = 1}^\infty q^{- l|b_1|}\int_{(\mathcal O^\times)^n\times(\mathcal O^\times)^m}|\pi^{c_k}f_\tau(x) +\pi^{d_k}g_\gamma(\pi^{lb_1}t)|^s|dxdt|\\
&= \sum_{l = 1}^\infty q^{- l|b_1|}\int_{(\mathcal O^\times)^n\times(\mathcal O^\times)^m}|\pi^{c_k}f_\tau(x) +\pi^{d_k+lm_g(b_1)}g_\gamma(t)|^s|dxdt|,
\end{align*}
in which the last equality occurs as $\gamma \subseteq F(b_1)$. If $lm_g(b_1) > 2c(\pi^{c_k}f_\tau,(\mathcal O^\times)^n) + 1$, then it implies from Corollary \ref{second-important} that
\begin{equation*}
\int_{(\mathcal O^\times)^n\times(\mathcal O^\times)^m}|\pi^{c_k}f_\tau(x) +\pi^{d_k+lm_g(b_1)}g_\gamma(t)|^s|dxdt| = \int_{(\mathcal O^\times)^n\times(\mathcal O^\times)^m}|\pi^{c_k}f_\tau(x)|^s|dxdt|.
\end{equation*}
From the latter, we can easily show that the second integral has the form $P_2(q^{-s})/(1 - q^{-1-s})$ for some $1$-variable rational polynomial $P_2$. The same argument also holds for the third integral. Finally, by the change of variables $x = \pi^{a_1}z,y = \pi^{b_1}t$, the last one can be computed as follows
\begin{align*}
& \int_{\{v(x) = ka_1,k\ge1\}\times\{v(y) = lb_1,l\ge1\}}|\pi^{c_k}f_\tau(x) +\pi^{d_k}g_\gamma(y)|^s|dxdy|\\ 
& \quad=  q^{-\mu_{k+1}s}\int_{\{v(x) = ka_1,k\ge1\}\times\{v(y) = lb_1,l\ge1\}}|f_\tau(x) +g_\gamma(y)|^s|dxdy|\\
&\quad = q^{-\mu_{k+1}s - |a_1| - |b_1|}\int_{\{v(x) = ka_1,k\ge0\}\times\{v(y) = lb_1,l\ge0\}}|\pi^{m_f(a_1)}f_\tau(z) +\pi^{m_g(b_1)}g_\gamma(t)|^s|dzdt|\\
&\quad = q^{-\mu_{k+1}s - \nu_{k+1}}Z_{c_{k+1},d_{k+1}}(s).
\end{align*}
The lemma follows.
\end{proof}
Back to the proof of Proposition \ref{important-computation}, we note that $c_1 = c_{1 + p}$, with $p$ defined in Lemma \ref{technicallemma}. Applying $p$ times Lemma \ref{ez-base} we get
\begin{equation*}
Z_{a_1,b_1}(s) = \frac{P(q^{-s})}{1 - q^{-1-s}} + q^{-(\sum_{k = 2}^{p+1} \mu_k)s - \sum_{k = 2}^{p+1}\nu_k}Z_{a_{1 +p},b_{1 + p}}(s),
\end{equation*}
where $P$ is a $1$-variable rational polynomial. From Lemma \ref{technicallemma}, we obtain
\begin{equation*}
Z_{a_1,b_1}(s) = \frac{P(q^{-s})}{(1 - q^{-1-s})(1 - q^{c_{a_1,b_1}(s)})}.
\end{equation*}
The same computation as in Lemma \ref{ez-base} gives us
\begin{equation*}
Z_{\Delta_1\times\Delta_2}(f_\tau \oplus g_\gamma;s) = \dfrac{H(q^{-s})}{(1 - q^{-1-s})} + q^{-|a_1|-|b_1|}Z_{a_1,b_1}(s),
\end{equation*}
where $H$ is a $1$-variable rational polynomial. This means that the proposition is proved in the case where $u=v=1$, $m_f(a_1)>0$ and $m_g(b_1) > 0$.

If $m_f(a_1) = 0$ and $m_g(b_1) > 0$, then
\begin{align*}
		Z_{\Delta_1\times\Delta_2}(f_\tau \oplus g_\gamma;s) &= \sum_{k,l = 0}^\infty\int_{\{v(x) = ka_1,\}\times\{v(y) = lb_1\}}|f_\tau(x) +g_\gamma(y)|^s|dxdy|\\
		&=\sum_{k = 0}^\infty \left(q^{-k|a_1|}\right)\sum_{l = 0}^\infty q^{- l|b_1|}\int_{(\mathcal O^\times)^n\times(\mathcal O^\times)^m}|f_\tau(z) + \pi^{lm_g(b_1)}g_\gamma(t)|^s|dzdt|.
	\end{align*}
	With $lm_g(b_1) > 2c(f_\tau,(\mathcal O^\times)^n) + 1$, from Corollary \ref{second-important}, we have
	\begin{equation*}
		\int_{(\mathcal O^\times)^n\times(\mathcal O^\times)^m}|f_\tau(z) + \pi^{lm_g(b_1)}g_\gamma(t)|^s|dzdt| = \int_{(\mathcal O^\times)^n\times(\mathcal O^\times)^m}|f_\tau(z)|^s|dzdt|.
	\end{equation*}
	Hence, there exists a rational polynomial $P$ such that
	\begin{equation*}
		Z_{\Delta_1\times\Delta_2}(f_\tau \oplus g_\gamma;s) = \dfrac{P(q^{-s})}{1 - q^{-1-s}}.
	\end{equation*}The case $m_f(a_1) > 0, m_g(b_1) = 0$ can be treated similarly as above.

	If $m_f(a_1) = m_g(b_1) = 0$, the proposition follows from Lemma \ref{first-important} and the following expansion
	\begin{align*}
		Z_{\Delta_1\times\Delta_2}(f_\tau \oplus g_\gamma;s) &= \sum_{k,l = 0}^\infty\int_{\{v(x) = ka_1,\}\times\{v(y) = lb_1\}}|f_\tau(x) +g_\gamma(y)|^s|dxdy|\\
		&=\left(\sum_{k,l = 0}^\infty q^{-k|a_1| - l|b_1|}\right)\int_{(\mathcal O^\times)^n\times(\mathcal O^\times)^m}|f_\tau(x) + g_\gamma(y)|^s|dxdy|.
	\end{align*}
	
	\textit{Induction hypothesis.} Suppose that the theorem is true for all Newton non-critical polynomials $f$ and $g$, for all faces $\tau$ and $\gamma$, for all cones $\Delta_1$ and $\Delta_2$ with the total number of generators does not exceed $u + v - 1$.
	
	Let us consider the first case where $m_f(a_1),m_g(b_1) > 0$. We define four sequences $(c_k)_{k \ge 1}$, $(d_k)_{k \ge 1}$, $(\mu_k)_{k\ge2},$ and $ (\nu_k)_{k\ge2}$ as in the base case. Put 
	$$Z_{c_k,d_k}(s) = Z_{\Delta_1\times \Delta_2}(\pi^{c_k}f_\tau \oplus \pi^{d_k}g_\gamma;s).$$ 
	
	\begin{lemma}\label{ez-induction}
		For all $k \ge 1$, we have
		\begin{equation*}
			Z_{c_k,d_k}(s) = \dfrac{Q_k(q^{-s})}{(1 - q^{-1-s})\prod_{(i,j)\ne(1,1)}(1 - q^{c_{a_i,b_j}(s)})} + q^{-\mu_{k+1}s - \nu_{k+1}}Z_{c_{k+1},d_{k+1}}(s),
		\end{equation*}
	where $Q_k$ is a $1$-variable rational polynomial.
	\end{lemma}

	\begin{proof}
		We are going to prove the lemma in the case $c_k = d_k$, the other cases are proved similarly. Define
		\begin{gather*}
			S_1 = \left\{\sum_{i=2}^uk_ia_i \mid k_2,\dots,k_u \in \mathbb N\right\},\hspace{5pt} S_2 = \left\{\sum_{i = 1}^uk_ia_i\mid  k_1 \ge 1, k_2,\dots,k_u \in \mathbb N\right\},\\
			T_1 = \left\{\sum_{j=2}^vl_jb_j \mid l_2,\dots,l_v \in \mathbb N\right\},\hspace{5pt} T_2 = \left\{\sum_{j = 1}^vl_jb_j\mid  l_1 \ge 1, l_2,\dots,l_u \in \mathbb N\right\}.
		\end{gather*}
	Then we have
	\begin{align*}
		Z_{c_k,d_k}(s)& = Z_{S_1\times T_1}(\pi^{c_k}f_\tau \oplus \pi^{d_k}g_\gamma;s) +  Z_{S_2\times T_1}(\pi^{c_k}f_\tau \oplus \pi^{d_k}g_\gamma;s)\\
		&\qquad + Z_{S_1\times T_2}(\pi^{c_k}f_\tau \oplus \pi^{d_k}g_\gamma;s) + Z_{S_2\times T_2}(\pi^{c_k}f_\tau \oplus \pi^{d_k}g_\gamma;s).
	\end{align*}
	Applying the induction hypothesis for functions $\pi^{c_k}f(x)$, $\pi^{d_k}g(y)$, for faces $\tau$, $\gamma$, and for $$S_1 = \bigoplus_{i = 2}^u\mathbb N a_i,\quad T_1 = \bigoplus_{j = 2}^v\mathbb N b_j,$$ 
	we have
	$$ Z_{S_1\times T_1}(\pi^{c_k}f_\tau \oplus \pi^{d_k}g_\gamma;s) = \dfrac{P_1(q^{-s})}{(1 - q^{-1-s})\prod_{i\ge 2,j\ge2}(1 - q^{c_{a_i,b_j}(s)})},$$
	for some rational polynomial $P_1$. By the change of variables  $x = \pi^{a_1}t$, we get
	\begin{equation*}
		Z_{S_2\times T_1}(\pi^{c_k}f_\tau \oplus \pi^{d_k}g_\gamma;s) = q^{-|a_1|}Z_{\Delta_1\times T_1}(\pi^{c_k + m_f(a_1)}f_\tau \oplus \pi^{d_k}g_\gamma;s).
	\end{equation*}
	By the induction hypothesis,
	\begin{equation*}
		Z_{\Delta_1\times T_1}(\pi^{c_k + m_f(a_1)}f_\tau \oplus \pi^{d_k}g_\gamma;s) = \dfrac{P_2(q^{-s})}{(1- q^{-1-s})\prod_{i\ge1,j\ge2}(1 - q^{c_{a_i,b_j}(s)})}
	\end{equation*}
	for some $1$-variable rational polynomial $P_2$. Therefore,
	\begin{equation*}
		Z_{S_2\times T_1}(\pi^{c_k}f_\tau \oplus \pi^{d_k}g_\gamma;s) = \dfrac{q^{-|a_1|}P_2(q^{-s})}{(1- q^{-1-s})\prod_{i\ge1,j\ge2}(1 - q^{c_{a_i,b_j}(s)})}.
	\end{equation*} Similarly, 
	\begin{equation*}
		Z_{S_1\times T_2}(\pi^{c_k}f_\tau \oplus \pi^{d_k}g_\gamma;s) = \dfrac{q^{-|b_1|}P_3(q^{-s})}{(1- q^{-1-s})\prod_{i\ge2,j\ge1}(1 - q^{c_{a_i,b_j}(s)})},
	\end{equation*}
	where $P_3$ is a $1$-variable rational polynomial. Finally, by the change of variables as $x = \pi^{a_1}z$ and $y= \pi^{b_1}t$, we get
	\begin{align*}
		Z_{S_2\times T_2}(\pi^{c_k}f_\tau \oplus\pi^{d_k}g_\gamma;s) &= q^{-\mu_{k+1}s}Z_{S_2\times T_2}(f_\tau \oplus g_\gamma;s)\\
		&=q^{-\mu_{k+1}s - |a_1|-|b_1|}Z_{\Delta_1\times\Delta_2}(\pi^{m_f(a_1)}f_\tau \oplus \pi^{m_g(b_1)}g_\gamma;s)\\
		&= q^{-\mu_{k+1}s - \nu_{k+1}}Z_{c_{k+1},d_{k+1}}(s).
	\end{align*}
	The lemma follows.
	\end{proof}
	The same computation as in the base case gives us the proposition when $m_f(a_1),m_g(b_1) > 0$. The either case $m_f(a_1) = 0$ or $m_g(b_1) = 0$ is treated similarly to the base by combining the change of variables formula and the induction hypothesis.
	\end{proof}


	\subsection{The $p$-adic zeta function} 
	Let $f(x) \in \mathcal O [x]$ and $g(y) \in \mathcal O [y]$ be polynomials, where $x=(x_1,\dots,x_n)$, $y=(y_1,\dots,y_m)$, let $\Gamma_f$ and $\Gamma_g$ be their Newton polyhedrons, respectively. For each facet of $\Gamma_f$, there exists a unique primitive vector perpendicular to it. Let $\mathcal D_f$ be the set of all such vectors, and we define $\mathcal D_g$ similarly for $g$. For $a \in \mathcal D_f$ and $b \in \mathcal D_g$, we denote by $\tau_a$ and $\gamma_b$ the facets corresponding to $a$ and $b$, respectively.
	
In this section, we are going to prove the following theorem,
\begin{theorem}\label{maintheorem}
Let $f(x) \in \mathcal O[x], g(y) \in \mathcal O [y]$ be two Newton non-critical polynomials. Then the Igusa local zeta function of $f \oplus g$ is equal to
\begin{equation*}
Z(f \oplus g ;s) = \dfrac{P(q^{-s})}{(1 - q^{-1-s})\prod_{a,b}(1 - q^{c_{a,b}(s)})},
\end{equation*}
where the product runs over all $(a,b)\in \mathcal D_f\times \mathcal D_g$, and  $P$ is a $1$-variable rational polynomial.
\end{theorem}

To prove the theorem, it is sufficient to show the following lemma.
	
\begin{lemma}\label{important-lemma}
Let $f(x) \in \mathcal O[x]$ and $ g(y) \in \mathcal O [y]$ be two Newton non-critical polynomials, let $\tau$ and $\gamma$ be proper faces of $\Gamma_f$ and $\Gamma_g$, respectively. Let $\{\tau_{a_1},\dots, \tau_{a_r}\}$ be all facets containing $\tau$, and let $\{\gamma_{b_1},\dots, \gamma_{b_t}\}$ be all facets containing $\gamma$. Assume that the sets $\{a_1,\dots,a_u\} \subseteq \{a_1,\dots,a_r\}$ and $\{b_1,\dots,b_v\} \subseteq \{b_1,\dots,b_t\}$ are $\mathbb{R}$-linearly independent sets. We put
\begin{align*}
\Delta_1:=  \bigoplus_{i = 1}^u \mathbb{N} a_i\quad \text{and}\quad \Delta_2 :=  \bigoplus_{j = 1}^v \mathbb{N} b_j.
\end{align*}
Then 
\begin{equation*}
Z_{\Delta_1 \times \Delta_2}(f \oplus g; s) = \dfrac{P(q^{-s})}{(1 - q^{-1-s})\prod_{i = 1}^u\prod_{j = 1}^v(1 - q^{c_{a_i,b_j}(s)})},
\end{equation*}
where $P$ is a $1$-variable rational polynomial.
\end{lemma}

\begin{proof}
We put 
$$\mathbf{I} = \{(A,B) \in \mathbb Z^2\mid  A \ge B \ge 2 \}$$ 
and provide it with the lexicographic order. The lemma is proved by induction on $(A,B) \in \mathbf{I}$, where $r + t = A, u + v = B$.
		
\textit{The base case} $ r = t = u = v = 1$. We first consider the case $m_f(a_1),m_g(b_1) > 0 $. Let $M = 2 c(f_\tau,(\mathcal O^\times)^n) + 2c(g_\gamma,(\mathcal O^\times)^m) + 2$. Define
\begin{gather*}
S_1 = \{ka_1: k\ge M\},\quad S_2 = \{ka_1: 0 \le k < M\}\\
T_1 = \{lb_1: l \ge M\},\quad T_2 = \{lb_1:0 \le l < M\}.
\end{gather*}
Then we have
\begin{equation}\label{expansion-base}
\begin{split}
Z_{\Delta_1 \times \Delta_2}(f\oplus g;s) = Z_{S_2\times T_2}(f\oplus g,s) &+ Z_{S_1\times T_2}(f\oplus g,s)\\
&+ Z_{S_2\times T_1}(f\oplus g,s) + Z_{S_1\times T_1}(f\oplus g,s).
\end{split}
\end{equation}
We have
\begin{align*}
Z_{S_2\times T_2}(f\oplus g,s) &= \sum_{k,l = 0}^{M - 1}\int_{\{v(x) = ka_1\}\times\{v(y) = lb_1\}} |f(x) + g(y)|^s |dxdy|\\
&= \sum_{k,l = 0}^{M - 1} q^{-k|a_1|-l|b_1|}\int_{(\mathcal O^\times)^n \times (\mathcal O^\times)^m } |f(\pi^{ka_1}x) + g(\pi^{lb_1}y)|^s |dxdy|.
\end{align*}
Since $f$ and $g$ are Newton non-critical, $f(\pi^{ka_1}x) + g(\pi^{lb_1}y)$ does not have any singular points in $(\mathcal O^\times)^n \times (\mathcal O^\times)^m$. By Lemma \ref{first-important}, 
$$\int_{(\mathcal O^\times)^n \times (\mathcal O^\times)^m } |f(\pi^{ka_1}x) + g(\pi^{lb_1}y)|^s |dxdy|=\frac{P(q^{-s})}{1 - q^{-1-s}},$$ 
where $P$ is a $1$-variable rational polynomial. So $Z_{S_2\times T_2}(f\oplus g,s)$ also has the same form. 
By the change of variables formula we have
\begin{equation*}
Z_{S_1\times T_2}(f\oplus g,s) = \sum_{l = 0}^{M - 1} q^{-l|b_1|}\sum_{k = M}^\infty q^{-k|a_1|} \int_{(\mathcal O^\times)^n \times (\mathcal O^\times)^m } |f(\pi^{ka_1}x) + g(\pi^{lb_1}y)|^s |dxdy|.
\end{equation*}
Note that
\begin{equation*}
f(\pi^{ka_1}x) = \pi^{km_f(a_1)}\hat{f}_k(x),
\end{equation*}
where $\hat{f}_k(x)$ is a polynomial in $\mathcal O[x]$. For $km_f(a_1) > 2c(g(\pi^{lb_1}y), (\mathcal O^\times)^m) + 1$, it implies from Corollary \ref{second-important} that
\begin{equation*}
\int_{(\mathcal O^\times)^n \times (\mathcal O^\times)^m } |f(\pi^{ ka_1}x) + g(\pi^{lb_1}y)|^s |dxdy| = \int_{(\mathcal O^\times)^n \times (\mathcal O^\times)^m } |g(\pi^{lb_1}y)|^s |dxdy|.
\end{equation*}
Hence, we get
\begin{equation*}
Z_{S_1\times T_2}(f\oplus g;s) = \dfrac{P(q^{-s})}{1 - q^{-1-s}},
\end{equation*}
where $P$ is a $1$-variable rational polynomial. The same argument also holds for $Z_{S_2\times T_1}(f\oplus
g;s)$. Finally, we have
\begin{align*}
Z_{S_1\times T_1}(f\oplus g;s) &= \sum_{k,l = M}^\infty \int_{\{v(x) =ka_1\}\times\{v(y) =lb_1\}}|f(x) + g(y)|^s|dxdy|\\
&= \sum_{k,l = M}^\infty q^{-k|a_1|-l|b_1|}\int_{(\mathcal O^\times)^n\times(\mathcal O^\times)^m}|f(\pi^{ka_1}z) + g(\pi^{lb_1}t)|^s|dzdt|.
\end{align*}
Since $ka_1 \in \Delta_\tau$, we have $f_\tau(\pi^{ka_1}z) = \pi^{ km_f(a_1)}f_\tau(z)$. For each $\omega \notin \tau = F_f(a_1)$, we have $\omega\cdot a_1 \ge m_f(a_1) + 1$, hence
\begin{equation*}
\omega\cdot(ka_1) \ge km_f(a_1) + k \ge  km_f(a_1) + M,
\end{equation*}
for all $\omega \notin \tau$ and $k \ge M$. Therefore, 
$$f_\tau(\pi^{ka_1}z) = \pi^{km_f(a_1)}(f_\tau(z) + \pi^M\hat{f}_\tau(z))$$ 
for some $\hat{f}_\tau(z) \in \mathcal O[z]$. Similarly, 
$$g(\pi^{lb_1}t) = \pi^{ lm_g(b_1)}(g_\gamma(t) + \pi^M\hat{g}_\gamma(t))$$ 
for some $\hat{g}_\gamma(t) \in \mathcal{O}[t]$. If $km_f(a_1) \ge lm_g(b_1)$, then by Corollary \ref{second-important}, we get
\begin{align*}
\int_{(\mathcal O^\times)^n\times(\mathcal O^\times)^m}&|f(\pi^{ka_1}z) + g(\pi^{lb_1}t)|^s|dzdt|\\
&=q^{-lm_g(b_1)s}\int_{(\mathcal O^\times)^n\times(\mathcal O^\times)^m}|\pi^{km_f(a_1) -  lm_g(b_1)}f_\tau(z) + g_\gamma(t) + \pi^Mh(z,t)|^s|dzdt|\\
&=q^{-lm_g(b_1)s}\int_{(\mathcal O^\times)^n\times(\mathcal O^\times)^m}|\pi^{km_f(a_1) - lm_g(b_1)}f_\tau(z) + g_\gamma(t)|^s|dzdt|\\
&=\int_{(\mathcal O^\times)^n\times(\mathcal O^\times)^m}|f_\tau(\pi^{ka_1}z) + g_\gamma(\pi^{lb_1}t)|^s|dzdt|,
\end{align*}
where $h(z,t)$ is in $\mathcal O [z,t]$. Therefore, we have
\begin{equation*}
\int_{(\mathcal O^\times)^n\times(\mathcal O^\times)^m}|f(\pi^{ka_1}z) + g(\pi^{lb_1}t)|^s|dzdt|=\int_{(\mathcal O^\times)^n\times(\mathcal O^\times)^m}|f_\tau(\pi^{ka_1}z) + g_\gamma(\pi^{lb_1}t)|^s|dzdt|.
\end{equation*}
Note that the above equality is also true when $km_f(a_1) \le lm_g(b_1)$, so we have $$Z_{S_1\times T_1}(f\oplus g;s) = Z_{S_1\times T_1}(f_\tau\oplus g_\gamma;s).$$
By the change of variables $x =\pi^{Ma_1}z$ and $y =\pi^{Mb_1}t$, we have
\begin{align*}
Z_{S_1 \times T_1}(f_\tau \oplus g_\gamma;s) &= q^{-M|a_1| -M|b_1|}Z_{\Delta_1 \times \Delta_2}(\pi^{Mm_f(a_1)}f_\tau \oplus \pi^{Mm_g(b_1)}g_\gamma;s)\\
 &= \frac{P(q^{-s})}{(1 - q^{-1-s})(1 - q^{c_{a_1,b_1}(s)}))},
\end{align*}
where $P$ is a $1$-variable rational polynomial. The last equality is obtained by applying Proposition \ref{important-computation} to $\pi^{Mm_f(a_1)}f$ and $\pi^{Mm_g(b_1)}g$. Thus, from (\ref{expansion-base}), the lemma follows.

\textit{Induction hypothesis.} Suppose that the lemma is valid for any pair of Newton smooth polynomials $(f,g)$, and for any pair of faces $(\tau,\gamma)$ together with any pair of cones $(\Delta_1,\Delta_2)$ satisfying the condition $(r + s, u + v) < (A , B)$. Now, let $f$ and $g$ be Newton non-critical polynomials, and assume that the faces and cones $\tau$, $\gamma$, $\Delta_1$ and $\Delta_2$ satisfy the condition $r + s = A, u + v = B$. We put $\tau' = \cap_{i = 1}^u \tau_{a_i}, \gamma' = \cap_{j = 1}^v \gamma_{b_j}$. If $\tau_a$ is a facet containing $\tau'$, then it also contains $\tau$, hence $a \in \{a_1,\dots,a_r\}$. Since $\tau'$ is the intersection of all facets containing it, $ \tau $ is a proper face of $\tau'$ if and only if the number of facets containing $\tau'$ is less than $r$. If $\tau \subsetneq \tau'$ or $\gamma \subsetneq \gamma'$, we can apply induction hypothesis for $f,g$ and $\tau',\gamma',
\Delta_1, \Delta_2$ and the lemma follows. When this is not the case, we have $\tau = \tau'$ and $\gamma = \gamma'$.
Put 
$$M = 2c(f_\tau,(\mathcal O^\times)^n) + 2c(g_\gamma,(\mathcal O^\times)^m) + 2.$$  
For $\varnothing \ne I = \{i_1,\dots,i_k\} \subseteq \{1,\dots,u\}$ and $r_I = (r_{i_1},\dots,r_{i_k})$ with $0\le r_{i_l}\le M - 1$, $1\leq l\leq k$, we put
\begin{align*}
S_{I,r_I}& = \left\{\sum_{i\in I}k_ia_i + \sum_{i \notin I}k_ia_i:k_i = r_i\text{ if }i\in I,  k_i \ge M \text{ if }i\notin I \right\},\\
S_\varnothing &= \left\{\sum_{i = 1}^u k_ia_i\mid  k_1,\dots,k_r\ge M \right\}.
\end{align*}
Similarly, we can define $T_{J,t_J}$ and $T_\varnothing$ for any $\varnothing \ne J = \{j_1,\dots,j_k\} \subseteq \{1,\dots,v\}$ and any $t_J = (t_{j_1},\dots,t_{j_k})$ with $0\le t_{j_l} \le M - 1$, $1\leq l\leq k$. We obtain following partitions of $\Delta_1$ and $\Delta_2$, respectively,
\begin{align*}
\Delta_1 = \bigsqcup_{I,r_I} S_{I,r_I} \cup S_\varnothing, \quad \Delta_2 = \bigsqcup_{J,t_J} T_{J,t_J} \cup T_\varnothing.
\end{align*}
Then
\begin{equation}\label{epansion-induction}
\begin{aligned}
Z_{\Delta_1\times\Delta_2}(f\oplus g;s) = \sum_{I,r_I,J,t_J}& Z_{S_{I,r_i}\times T_{J,t_J}}(f \oplus g;s) + \sum_{I,r_i}Z_{S_{I,r_I}\times T_\varnothing}(f\oplus g;s)\\
&\quad +\sum_{J,t_J}Z_{S_\varnothing\times T_{J,t_J}}(f\oplus g;s) + Z_{S_\varnothing\times T_\varnothing}(f\oplus g;s).
\end{aligned}
\end{equation}
For nonempty sets $I$ and $J$, we consider the following change of variables
\begin{equation*}
x = \pi^{\sum_{i\in I}r_ia_i + M\sum_{i\notin I}a_i}z, \quad  y = \pi^{\sum_{j \in J}t_jb_j + M\sum_{j\notin J}b_j}t.
\end{equation*}
Then we have
\begin{align*}
&Z_{S_{I,r_i}\times T_{J,t_J}}(f\oplus g;s) \\
&\qquad\quad = q^{-\sum_{i\in I}r_i|a_i| - M\sum_{i\notin I}|a_i| - \sum_{j\in J}t_j|b_j| - M\sum_{j\notin J}|b_j|}Z_{\Delta_{1,I}\times\Delta_{2,J}}(f_{I,r_I}\oplus g_{J,t_J};s),
\end{align*}
where 
\begin{align*}
\Delta_{1,I}:= \bigoplus_{i \notin I} \mathbb{N} a_i, \quad \Delta_{2,J} := \bigoplus_{j \notin J} \mathbb{N} b_j,
\end{align*}
and 
\begin{align*}
f_{I,r_I}(z)& = f(\pi^{\sum_{i\in I}r_ia_i + M\sum_{i\notin I}a_i}z),\\ g_{J,t_J}(t)& = g(\pi^{\sum_{j \in J}t_jb_j + M\sum_{j\notin J}b_j}t).	
\end{align*}
Applying the induction hypothesis to $f_{I,r_I}$, $g_{J,t_J}$, $\tau_I := \cap_{i\notin I}\tau_{a_i}$ and $\gamma_J := \cap_{j \notin J}\gamma_{b_j}, \Delta_{1,I},\Delta_{2,J}$, we have
\begin{equation*}
Z_{\Delta_{1,r_I}\times\Delta_{2,t_J}}(f_{I,r_I}\oplus g_{J,t_J};s) = \dfrac{P(q^{-s})}{(1 - q^{-1-s})\prod_{i\notin I,j \notin J}(1 - q^{c_{a_jb_j}(s)})},
\end{equation*} 
for some $1$-variable rational polynomial $P$. Therefore, $Z_{S_{I,r_i}\times T_{J,t_J}}(f\oplus g;s)$ also has the same form. A similar argument also holds for either $I$ or $J$ empty but not for both empty. We have
\begin{align*}
Z_{S_\varnothing\times T_\varnothing}(f\oplus g;s) &= \sum_{k_1,\dots,k_r = M}^\infty\sum_{l_1,\dots,l_t = M}^\infty \int_{\{v(x) = \sum k_ia_i\}\times\{v(y) = \sum l_jb_j\}}|f(x) + g(y)|^s|dxdy|.
\end{align*}
By the change of variables formula,
\begin{align*}
&\int_{\{v(x) = \sum k_ia_i\}\times\{v(y) =\sum l_jb_j\}}|f(x) + g(y)|^s|dxdy|\\
&\qquad\quad =q^{-\sum_{i = 1}^u k_i|a_i| - \sum_{l = 1}^v l_j|b_j|}\int_{(\mathcal O^\times)^n\times(\mathcal O^\times)^m}|f(\pi^{\sum_{i = 1}^u k_ia_i}z) + g(\pi^{\sum_{j = 1}^v l_jb_j}t)|^s|dzdt|.
\end{align*}
As $a_i \in \overline{\Delta}_\tau$ for all $i$, we have
\begin{equation*}
f(\pi^{\sum_{i = 1}^u k_ia_i}z) = \pi^{\sum_{i = 1}^u k_im_f(a_i)}f_\tau(z) + \sum_{\omega \notin \tau} a_\omega\pi^{(\sum_{i = 1}^u k_ia_i)\cdot\omega}z^\omega.
\end{equation*}
Moreover, for any $\omega \notin \tau = \cap_{i = 1}^u \tau_{a_i}$, there is an $i_\omega \in \{1,\dots,u\}$ such that $\omega \notin \tau_{a_{i_\omega}} = F_f(a_{i_\omega})$. Therefore 
\begin{equation*}
(\sum_{i = 1}^u k_ia_i)\cdot\omega \ge \sum_{i = 1}^u k_i m_f(a_i) + k_{i_\omega} \ge \sum_{i = 1}^u k_i m_f(a_i) + M,
\end{equation*}
hence 
$$f(\pi^{\sum_{i = 1}^u k_ia_i}z) = \pi^{\sum_{i = 1}^u k_im_f(a_i)}(f_\tau(z) + \pi^M \hat{f}_\tau(z))$$ 
for some $\hat{f}_\tau(z) \in \mathcal O[z]$. Similarly, we also have 
$$g(\pi^{\sum_{j = 1}^v l_jb_j}t) = \pi^{\sum_{j = 1}^v l_jm_g(b_j)}(g_\gamma(t) + \pi^M\hat{g}_\gamma(t)).$$ 
The same argument as in the base case gives us 
$$Z_{S_\varnothing\times T_\varnothing}(f\oplus g;s) = Z_{S_\varnothing\times T_\varnothing}(f_\tau \oplus g_\gamma;s).$$ 
By change of variables $x = \pi^{M\sum_{i = 1}^u a_i}z$ and $y = \pi^{M \sum_{j = 1}^v b_j}t$, we get
\begin{align*}
Z_{S_\varnothing\times T_\varnothing}(f_\tau \oplus g_\gamma;s) &= q^{-M\sum_{i = 1}^u |a_i| - M\sum_{j = 1}^v |b_j|}Z_{\Delta_1\times \Delta_2}(f_\tau(\pi^{M\sum_{i = 1}^u a_i}z)+  g_\gamma(\pi^{M \sum_{j = 1}^v b_j}t);s)\\
&= q^{-M\sum_{i = 1}^u |a_i| - M\sum_{j = 1}^v |b_j|}Z_{\Delta_1\times \Delta_2}(\pi^{M\sum_{i = 1}^u m_f(a_i)}f_\tau \oplus  \pi^{M \sum_{j = 1}^v m_g(b_j)}g_\gamma;s)\\
&= \dfrac{P(q^{-s})}{(1 - q^{-1-s})\prod_{i = 1}^u\prod_{j = 1}^v(1 - q^{c_{a_i,b_j}(s)})},
\end{align*}
for some $1$-variable rational polynomial $P$. The last equality follows from Proposition \ref{important-computation} for $\pi^{M\sum_{i = 1}^u m_f(a_i)}f_\tau(x)$ and $\pi^{M \sum_{j = 1}^v m_g(b_j)}g_\gamma(y)$. The lemma now follows from (\ref{epansion-induction}).
\end{proof}

\begin{corollary}\label{main-cor}
Let $f(x) \in \mathcal O[x]$ and $g(y) \in \mathcal O [y]$ be Newton smooth polynomials, let $\tau$ and $\gamma$ be proper faces of $\Gamma_f$ and $\Gamma_g$, respectively. Let $\{\tau_{a_1},\dots, \tau_{a_r}\}$ be all facets containing $\tau$ and $\{\gamma_{b_1},\dots, \gamma_{b_t}\}$ be all facets containing $\gamma$. Then 
\begin{equation*}
Z_{\Delta_\tau \times \Delta_\gamma}(f\oplus g; s) = \dfrac{Q(q^{-s})}{(1 - q^{-1-s})\prod_{i = 1}^r\prod_{j = 1}^t(1 - q^{-c_{a_i,b_j}(s)})},
\end{equation*}
where $Q$ is a $1$-variable rational polynomial.
\end{corollary}

\begin{proof}
There exists a finite partition of $\Delta_\tau$ into cones $\Delta_\tau^i$ such that $\Delta_\tau^i$ is spanned by the  $\mathbb R$-linearly independent set $\{a_1,\dots,a_u\} \subseteq \{a_1,\dots,a_r\}$ . Similarly, each $\Delta_\gamma$ can be partitioned into cones $\Delta_\gamma^j$ such that $\Delta_\gamma^j$ is spanned by the  $\mathbb R$-linearly independent set $\{b_1,\dots,b_v\} \subseteq \{b_1,\dots,b_t\}$. Note that, each $\Delta_\tau^i \cap \mathbb N^n$ is the disjoint union of the sets $a + \bigoplus_{i = 1}^u \mathbb N a_i$, where $a$ runs through all elements of
\begin{equation*}
\mathbb N^n \cap \left\{\sum_{i =1 }^uk_ia_i: 0<k_i\le 1\right\}.
\end{equation*}
We also obtain a similar partition for $\Delta_\gamma^j$. We have,
\begin{equation}\label{ExpansionOfCone}
\begin{split}
Z_{\Delta_\tau\times\Delta_\gamma}(f\oplus g;s) &= \sum_{i,j} Z_{\Delta_\tau^i\times\Delta_\gamma^j}(f\oplus g;s)\\
&= \sum_{i,j} \sum_{a,b} Z_{\{a + \bigoplus_{i = 1}^u \mathbb N a_i\}\times\{b + \bigoplus_{j = 1}^v \mathbb N b_j\}}(f\oplus g;s).
\end{split}
\end{equation}
By the change of variables $x = \pi^az$ and $y = \pi^bt$, we have
\begin{equation*}
Z_{\{a + \bigoplus_{i = 1}^u \mathbb N a_i\}\times\{b + \bigoplus_{j = 1}^v \mathbb N b_j\}}(f\oplus g;s) = q^{-|a| - |b|}Z_{\{\bigoplus_{i = 1}^u \mathbb N a_i\}\times\{\bigoplus_{j = 1}^v \mathbb N b_j\}}(f(\pi^az) + g(\pi^bt);s).
\end{equation*}
The corollary follows by applying Lemma \ref{important-lemma} for $f(\pi^az)$ and $g(\pi^bt)$ and the expansion (\ref{ExpansionOfCone}).
\end{proof}

We now prove the main theorem.

\begin{proof}[Proof of Theorem \ref{maintheorem}]
We have the following partitions
\begin{equation*}
(\mathbb{R}^+)^n = \{0\}\cup \bigcup_{\tau}\Delta_\tau, \quad (\mathbb{R}^+)^m = \{0\}\cup \bigcup_{\gamma}\Delta_\gamma,
\end{equation*}
where $\tau,\gamma$ runs through all faces of $\Gamma_f,\Gamma_g$ respectively. Therefore,
\begin{align*}
Z(f\oplus g;s) &= \sum_{k \in \mathbb N^n}\sum_{l \in \mathbb N^m}\int_{\{v(x) = k\}\times\{v(y) = l\}}|f(x) + g(y)|^s|dxdy|\\
&= \int_{(\mathcal O^\times)^n\times(\mathcal O^\times)^m}|f(x) + g(y)|^s|dxdy| +\int_{(\mathcal O^\times)^n\times\mathcal O^n}|f(x) + g(y)|^s|dxdy|\\
&\qquad \qquad \qquad +\int_{\mathcal O^n\times(\mathcal O^\times)^m}|f(x) + g(y)|^s|dxdy| + \sum_{\tau,\gamma}Z_{\Delta_\tau\times\Delta_\gamma}(f\oplus g;s).
\end{align*}
Thus Theorem \ref{maintheorem} follows by applying Lemma \ref{first-important} for three first integrals and Corollary \ref{main-cor} for the last one.
\end{proof}

\begin{ack}
The first author thanks the Vietnam Institute for Advanced Study in Mathematics (VIASM) for warm hospitality during his visit. He would also like to acknowledge support from the ICTP through the Associates Programme (2020-2025). 
\end{ack}

\end{document}